\documentclass[12pt]{amsart}

\usepackage[utf8]{inputenc}

\usepackage[utf8]{inputenc}
\usepackage{amssymb}
\usepackage{graphicx}
\usepackage{graphicx,color}
\usepackage{latexsym}
\usepackage[all]{xy}
\usepackage{mathrsfs}
\usepackage{enumerate}
\usepackage{multicol}
\usepackage{lipsum}
\usepackage{amssymb}
\usepackage{tikz}
\usepackage{float}
\usepackage{bm}
\usepackage{mathptmx}
\usetikzlibrary{shapes.geometric}

\usepackage{amsmath}
\usepackage{amsfonts}
\usepackage{amssymb,enumerate}
\usepackage{amsthm}
\usepackage{fancyhdr}
\usepackage{hyperref}
\usepackage{cleveref}
\usepackage{xcolor}
\usepackage{enumitem}
\usepackage{float}
\usepackage{verbatim}
\usepackage{tikz}
\usetikzlibrary[patterns]
\usetikzlibrary{arrows}
\usepackage{mathrsfs}
\usetikzlibrary{arrows}
\usepackage{listings}
\theoremstyle{plain}
\newtheorem{proposition}{Proposition}
\newtheorem{theorem}{Theorem}

\newtheorem{corollary}{Corollary}
\newtheorem{conj}{Conjecture}
\newtheorem{prob}{Problem}

\theoremstyle{definition}
\newtheorem{defin}{Definition}
\newtheorem{question}{Question}

\newtheorem{ex}{Example}
\theoremstyle{remark}
\newtheorem{rem}[theorem]{Remark}

\numberwithin{equation}{section}
\usepackage[backend=biber,style=numeric,doi=false,isbn=false,url=false]{biblatex}
\begin{document}
	
	\title[On the quotient of Milnor and Tjurina numbers]{On the quotient of Milnor and Tjurina numbers for two-dimensional isolated hypersurface singularities}
	
	
	\author[P. Almir\'{o}n]{Patricio Almir\'{o}n}
	\address{Instituto de Matemática Interdisciplinar (IMI), Departamento de \'{A}lgebra, Geometr\'{i}a y Topolog\'{i}a\\
		Facultad de Ciencias Matem\'{a}ticas\\
		Universidad Complutense de Madrid\\
		28040, Madrid, Spain.
	}
	\email{palmiron@ucm.es}
	
	\thanks{The author is supported by Spanish Ministerio de Ciencia, Innovaci\'{o}n y Universidades MTM2016-76868-C2-1-P.   }
	
	
	

	\begin{abstract}
		In this paper we give a complete answer to a question posed by Dimca and Greuel about the quotient of the Milnor and Tjurina numbers of a plane curve singularity. We put this question into a general framework of the study of the difference of Milnor and Tjurina numbers for isolated complete intersection singularities showing its connection with other problems in singularity theory.
		
	\end{abstract}
	\maketitle

\section{Introduction}

Let \((X,\mathbf{0})\subset(\mathbb{C}^N,0)\) be a germ of an isolated hypersurface singularity defined by an equation \(f\in\mathcal{O}_{(\mathbb{C}^N,0)}\). For such singularities there are two important invariants: the Milnor number \(\mu,\) and the Tjurina number \(\tau.\) Those numbers can be expressed as:

\[\mu:=\dim_{\mathbb{C}}\frac{\mathbb{C}\{x_1,\dots,x_N\}}{(\frac{\partial f}{\partial x_1},\dots ,\frac{\partial f}{\partial x_N})},\quad\tau:=\dim_{\mathbb{C}}\frac{\mathbb{C}\{x_1,\dots,x_N\}}{(f,\frac{\partial f}{\partial x_1},\dots ,\frac{\partial f}{\partial x_N})}.\]

By definition, it is trivial that \(\mu-\tau\geq 0\). In fact, it is a well known result by K. Saito \cite{saitohom} that \(\mu-\tau=0\) if and only if the hypersurface singularity is quasihomogeneous.

For isolated complete intersection singularities (ICIS) of dimension \(n=N-r\) defined by an ideal \(\mathcal{I}=(f_1,\dots,f_r)\), Hamm \cite[Satz 1.7]{hamm} was the first to show that the Milnor fiber of \((X,0)\) is homotopy equivalent to a bouquet of spheres, extending the previous results of Milnor \cite{milnorbook}. In contrast, the algebraic definition of the Milnor number of an ICIS through the partial derivatives of the defining equations is due to the works of Greuel \cite{greuel,greuel2} and L\^{e} \cite{lemilnor} independently.  On the other hand, Tjurina's work \cite{tjurinaoriginal} identifies \(\operatorname{Ext}^{1}_{\mathcal{O}_{(X,0)}}(\Omega^1_{(X,0)},\mathcal{O}_{(X,0)})\) as the base space of the miniversal deformation of a normal isolated singularity with \(\operatorname{Ext}^{2}_{\mathcal{O}_{(X,0)}}(\Omega^1_{(X,0)},\mathcal{O}_{(X,0)})=0\). In her honor, Greuel in \cite{greuel} named the dimension of that base space as Tjurina number. Thus, the Milnor and Tjurina numbers of an ICIS can be defined as

\[\mu:=\operatorname{rk} H_{n}(F),\quad \tau:=\dim_{\mathbb{C}}(\operatorname{Ext}^{1}_{\mathcal{O}_{(X,0)}}(\Omega^1_{(X,0)},\mathcal{O}_{(X,0)})),\]
where \(F\) is the Milnor fiber of the smoothing of \((X,0)\), \(\mathcal{O}_{(X,0)}=\mathbb{C}\{x_1,\dots,x_N\}/\mathcal{I}\) and \(\Omega^1_{(X,0)}\) is the corresponding module of differential \(1\)--forms at \((X,0)\). In this case, the inequality \(\mu-\tau\geq 0\) was proven by Greuel \cite{greuel} in case of dimension \(n=1\) and by Looijenga and Steenbrink \cite{looi} in dimension bigger or equal than \(2\).

Despite the results concerning the inequality \(\mu-\tau\geq 0\), very few is known in the literature concerning sharp upper bounds for \(\mu-\tau\) of the form \(C\mu\) with \(C\in\mathbb{Q}.\) In the hypersurface case, as far as the author knows, Liu \cite{liu} is the first who provided some bounds of this type (see for example Proposition \ref{prop:liu}). One of the goals of this paper is to motivate the study of the following problem by showing its connection with other problems in singularity theory:
\begin{prob}\label{problem}
	Let \((X,0)\subset(\mathbb{C}^N,0)\) be an isolated complete intersection singularity of dimension \(n\) and codimension \(r=N-n\). Is there an optimal \(\frac{b}{a}\in\mathbb{Q}\) with \(b<a\) such that 
	\[\mu-\tau<\frac{b}{a}\mu\;\text{?}\]
	Where optimal means that there exist a family of singularities such that \(\mu/\tau\)  tends to \(\frac{a}{a-b}\) when the multiplicity at the origin tends to infinity.
\end{prob}
The main results of this paper are a complete answer to this Problem in the case \(N=2,\,r=1\) and partial answers in the cases \(N=3,\,r=1\) and \(r=N-1\) with arbitrary \(N\).

The case of plane curve singularities, i.e. \(N=2,\,r=1\) connects our problem with the following question posed by Dimca and Greuel in 2017:
\begin{question}\cite[Question 4.2]{dim} \label{conjecture}
	Is it true that $\mu/\tau < 4/3$ for any isolated plane curve singularity?
\end{question}
The guessed bound is inferred by Dimca and Greuel from some families of plane curve singularities that asymptotically achieve the \(4/3\) bound (see \cite[Ex. 4.1]{dim}). From this point of view, Question \ref{conjecture} can be split into two questions: Is it true? If it is true, is there an intrinsic reason for the \(4/3\) bound?

There are some partial answers to Dimca and Greuel's Question \ref{conjecture} given by Blanco and the author \cite{alblanc}, Alberich-Carrami\~{n}ana, Blanco, Melle-Hern\'{a}ndez and the author \cite{taumin}, Genzmer and Hernandes \cite{genzmertau} and Wang \cite{wang} (see Section \ref{sec:dgprob} for a more detailed description of those results). All of them show that \(\mu/\tau\) satisfies the \(4/3\) bound for some special families of plane curve singularities. Despite those partial positive answers to Question \ref{conjecture}, there is no clue in Wang's results \cite{wang} nor in Genzmer-Hernandes \cite{genzmertau} neither in our first results \cite{alblanc,taumin} as to whether the numbers \(3,4\) can be inferred from deeper arguments than just explicit computations, i.e. is there an intrinsic reason to consider the invariant \(3\mu-4\tau\) instead of any other combination of the type \(a\mu-b\tau\) with \((a,b)\neq (3,4)?\)

Here we are not only going to give a positive answer to Dimca and Greuel's Question \ref{conjecture} in its full generality but also we are going to provide an intrinsic reason for the \(4/3\) bound. Our answer to Question \ref{conjecture} (Theorem \ref{teorema4/3}) solve both questions through the study of \(\mu/\tau\) for the normal two-dimensional double point \(\{z^2=f(x,y)\}.\) In particular, our proof (Theorem \ref{teorema4/3}) shows that the \(4/3\) bound is inferred from the restriction for the number of adjunction conditions of a normal two-dimensional double point singularity. Thus, we provide a solution to Problem \ref{problem} in the case \(N=2,\,r=1\). Moreover, as one can see, our point of view is completely new from the techniques used in \cite{taumin,alblanc,genzmertau,wang} to solve Question \ref{conjecture} for some special cases.

As a consequence of our approach, we can use a result of Teissier \cite{teissier-appendix} to show that if \((C,\mathbf{0})\) is an irreducible germ of curve, not necessarily plane, with the semigroup of a plane curve singularity (See Sec. \ref{sec:monomialcurve} for a precise definition) then \(\mu-\tau<\mu/4\) (Corollary \ref{cor:mu4space}). This fact constitutes a partial answer to Problem \ref{conjecture} in the case \(r=N-1\) with arbitrary \(N\).

To finish, we move to the case \(N=3\) and \(r=1\). In this case, Wahl \cite{wahl} proves that \(\mu-\tau\leq 2p_g\), where \(p_g\) is the geometric genus of the singularity, i.e. if \(\widetilde{X}\rightarrow X\) is a resolution of singularities of \(X,\) \(p_g:=\dim_{\mathbb{C}} H^{1}(\widetilde{X},\mathcal{O}_{\widetilde{X}}).\) This bound connects Problem \ref{problem} with the following long standing and widely studied conjecture posed by Durfee in 1978:
\begin{conj}\cite[Conjecture 5.3]{durfee}\label{durfeeconj}
	For any isolated hypersurface singularity \((X,0)\subset(\mathbb{C}^3,0)\)
	\[6p_g\leq\mu.\]
\end{conj}

As we will see in Section \ref{sec:durfee}, the cases where Durfee's conjecture holds motivate us to propose \(\mu/3\) as the optimal bound for the case  \(N=3\) and \(r=1\) of Problem \ref{problem}(Proposition \ref{boundsurface}). We conjecture that this bound is true for any isolated hypersurface singularity in \(\mathbb{C}^3\) (Conjecture \ref{conjetura3/2}). Moreover, Durfee's conjecture \ref{durfeeconj} implies Conjecture \ref{conjetura3/2}. In this way, Conjecture \ref{conjetura3/2} provides an easy criterion to check the validity of Durfee's conjecture \ref{durfeeconj}. 

From this point of view, we think that the general setting which provides Problem \ref{problem} can be useful in the understanding and resolution of other problems in singularity theory.

\vskip 4mm
\textbf{Acknowledgments.} The author would like to thank Maria Alberich-Carrami\~{n}ana and Alejandro Melle-Hern\'{a}ndez for their constant support and the helpful comments and suggestions.

\section{Remarks on the difference \(\mu-\tau\)}\label{sec:poincare}

Let \((X,\mathbf{0})\) be a germ of isolated hypersurface singularity. Following \cite{yau}, let 
\[0\rightarrow\mathbb{C}\rightarrow\mathcal{O}_{(X,\mathbf{0})}\xrightarrow{d^{0}}\Omega^{1}_{(X,\mathbf{0})}\xrightarrow{d^{1}}\Omega^{2}_{(X,\mathbf{0})}\rightarrow\cdots \]
be the Poincar\'{e} complex at \(\mathbf{0}\), where \(\Omega^{p}_{(X,\mathbf{0})}\) is the sheaf of differential \(p\)--forms and \(d^p\) the usual differential operator. The Poincar\'{e} numbers of \(X\) at \(\mathbf{0}\) are defined as 
\[p^{(i)}:=\dim_{\mathbb{C}}\frac{\operatorname{Ker}d^i}{\operatorname{Im}d^{i-1}}\quad\text{for all}\; i\geq 0.\]
If \((X,\mathbf{0})\) is a hypersurface singularity of dimension \(n\) then Brieskorn \cite{brieskorn} proved \(p^{(i)}=0\) if \(i\leq n-2\) and Sebastiani \cite{sebastiani} proved \(p^{(n-1)}=0\). After that, Saito \cite{saitohom} proved that the Poincar\'{e} complex is exact if and only if the singularity is quasihomogeneous. From this, one has
\begin{theorem}\cite{saitohom}
	If \((X,\mathbf{0})\) is a germ of an isolated hypersurface singularity then \(p^{(n)}=\mu-\tau\). Even more, \(p^{(n)}=0\) if and only if \((X,\mathbf{0})\) is quasihomogeneous.
\end{theorem}
In 1975, Greuel \cite{greuel2} extended the results of Brieskorn and Sebastiani to the case of isolated complete intersection of any dimension and proved that if the singularity is quasihomogeneous then \(\mu=\tau\). The converse was proven by Wahl \cite{wahl} in the case of dimension \(2\) and for any dimension by Vosegaard \cite{Vosegaard}. However, in this full generality the exactness of the Poincar\'{e} complex does not imply the quasihomogeneity of the singularity as Pfister and Sch\"{o}nemann show in \cite{pfister}.  In this way, the difference \(\mu-\tau\) or alternatively the quotient \(\mu/\tau\) must be considered as a measure of the non-quasihomogeneity of the singularity and not as a measure of the exactness of the Poincar\'{e} complex. 

\begin{rem}
	We refer to \cite[Sections 7.2.4, 7.2.6]{greuel-survey} for a survey by Greuel on Milnor versus Tjurina number, not only for complete intersections, but also for reduced curve singularities.
\end{rem}

To give formulas for the difference \(\mu-\tau\) in terms of other invariants of the singularity is, in general, a difficult task. For example, for isolated complete intersections of dimension \(n\geq 2\), in 1985 \cite{looi} Looijenga and Steenbrink give a precise formula for this difference in terms of the mixed Hodge structure of the singularity:

\begin{theorem}\cite{looi}\label{teolooi}
	If \((X,x)\) is an isolated complete intersection singularity of dimension \(n\geq 2,\) then \[\mu-\tau=\sum_{p=0}^{n-2}h^{p,0}(X,x)+a_1+a_2+a_3,\]
	where \(h^{p,q}(X,x)\) denotes the \((p,q)\)--Hodge number of the mixed Hodge structure which is naturally defined on the local cohomology group \(H^{n}(X,X-\{x\};\mathbb{C})\) and the numbers \(a_1,a_2,a_3\) are nonnegative invariants of a resolution of \((X,x).\)
\end{theorem}

Despite this formula, not so much can be said about upper bounds for \(\mu-\tau\) as the one proposed in Problem \ref{problem}. By using purely algebraic methods, a first approach is to use the Brian\c{c}on--Skoda Theorem \cite{brianon} as it was showed by Liu in \cite{liu}:
\begin{proposition}\cite[Theorem 1.1]{liu}\label{prop:liu}
	For any germ of isolated hypersurface singularity defined by a germ of function \(f:\mathbb{C}^n\rightarrow\mathbb{C}\) we have \[\mu-\tau\leq \frac{n-1}{n}\mu.\] 
\end{proposition}
However, this bound is far from being sharp as we will show with Theorem \ref{teorema4/3} (see Sec.\ref{sec:dgprob}).

In the case of isolated complete intersection singularities of dimension \(2\), an alternative formula to the previous one of Looijenga and Steenbrink (Theorem \ref{teolooi}) was proven by Wahl \cite{wahl}. From his formula, Wahl can obtain the following upper bound.
\begin{theorem} \cite[Cor. 2.9]{wahl}
	\label{wahl} Let \((X,0)\) be a germ of isolated complete intersection singularity of dimension \(2\). Then
	\[\mu-\tau\leq 2p_g-\dim H^1(A;\mathbb{C}),\] 
	where \(A\) is the exceptional divisor of a minimal good resolution of \(X.\) 
\end{theorem}
Moreover, Example 4.6 of \cite{wahl} shows that this bound is sharp if one takes a generic positive weight deformation of 
\[z^2+x^{2a+1}+y^{2a+2}=0.\]
It would be certainly interesting to classify the surface singularities with maximal \(\mu-\tau\). This leads us to ask the following question
\begin{question}\label{quest:2}
	Which hypersurface singularities \((X,0)\subset(\mathbb{C}^3,0)\) satisfy \(\mu-\tau=2p_g-\dim H^1(A;\mathbb{C})?\)
\end{question}

\section{Dimca and Greuel problem for plane curve singularities}\label{sec:dgprob}

The first result about Question \ref{conjecture} is given for semi-quasi-homogeneous singularities in 2018 by Blanco and the author \cite{alblanc}. In April 2019 three different answers for irreducible plane curve singularities appeared in a short time. Alberich-Carrami\~{n}ana, Blanco, Melle-Hern\'{a}ndez and the author in \cite{taumin} give a positive answer to Question \ref{conjecture} for irreducible plane curve singularities through a formula for the minimal Tjurina number in an equisingularity class in terms of the sequence of multiplicities. A few days later, Genzmer and Hernandes in \cite{genzmertau} provide an alternative proof of Dimca and Greuel's inequality for the irreducible plane curve case. Despite the fact that both papers use quite different techniques, both are based on the explicit computations about the moduli space of an irreducible plane curve singularity given by Genzmer in \cite{genzmer16}. Finally, Wang in \cite{wang} gives another alternative proof for the irreducible case based also in Genzmer's result about the dimension of the generic component of the moduli space \cite{genzmer16}. Moreover, Wang's approach is of different nature since he proves that \(3\mu-4\tau\) satisfy a certain property (monotonicity under blow ups) which provides a nice perspective in the possible applications of Dimca and Greuel's Question \ref{conjecture} in the irreducible case.

However, none of the methods used in those proofs allow to answer the question: can the \(4/3\) bound be inferred from a deeper argument than just explicit computation of these invariants? Here, we are going to not only give a positive answer to Dimca and Greuel's Question \ref{conjecture} in its full generality but also a non computational explanation for the \(4/3\) bound.

To do so, let us consider the equation \(f\in\mathbb{C}\{x,y\}\) of a germ of isolated plane curve singularity in \((\mathbb{C}^2,\mathbf{0}).\) Now, we consider the germ of isolated hypersurface singularity  \((X,\mathbf{0})\subset(\mathbb{C}^3,\mathbf{0})\) defined by

\[F(x,y,z)=f(x,y)+z^2=0.\]
For such a singularity, the geometric genus has the following upper bound proved by Tomari:

\begin{theorem}\cite[Thm. A]{tomari} \label{thm:tomari} Let \((X,\mathbf{0})\subset(\mathbb{C}^3,\mathbf{0})\) be a germ of isolated hypersurface singularity defined by an equation \(F(x,y,z)=z^2+f(x,y)\) with \(f(x,y)\) of order at least two. Then \[8p_g+1\leq\mu.\]
\end{theorem}

Now we are ready to provide a full answer to Dimca and Greuel Question \ref{conjecture}.

\begin{theorem}\label{teorema4/3}
	For any germ of a plane curve singularity 
	\[\frac{\mu}{\tau}<\frac{4}{3}.\]
\end{theorem}

\begin{proof}
	
	Let \(\xi:f(x,y)=0\) be a germ of isolated plane curve singularity. Let us consider the germ of double point \((X,0)\) defined by the equation \[F(x,y,z)=f(x,y)+z^2=0.\] 
	It is then trivial to check that the Tjurina ideal of \((X,0)\) can be expressed as  \[\Big(f,\frac{\partial f}{\partial x},\frac{\partial f}{\partial y},z\Big)\subset\mathbb{C}\{x,y,z\}.\] Then it is obvious that the Tjurina number of the double point \(\tau(X)\) is equal to the Tjurina number of the germ of plane curve defined by \(f(x,y)=0\).
	
	Let \(p_g\) be the geometric genus of the double point \(X\). From Tomari's Theorem \ref{thm:tomari} we know that \(p_g<\mu/8.\) Combining this with Wahl's Theorem \ref{wahl} gives 
	\[\mu(\xi)-\tau(\xi)=\mu(X)-\tau(X)\leq2p_g<\mu/4\Rightarrow\frac{\mu}{\tau}<\frac{4}{3}.\]
\end{proof}

In this way, we can conclude that the bound \(4/3\) for the quotient \(\mu/\tau\) of plane curve singularities is inferred from the rich properties of the geometric genus of the corresponding normal two-dimensional double point singularity. More concretely, recall that Merle and Teissier \cite[Section 1]{merleteissier} showed that the geometric genus is the number of adjunction conditions imposed by the singularity. Therefore, we can conclude that the bound \(4/3\) is due to the restrictions for the adjunction conditions of a normal two-dimensional double point singularity.

\subsection{Curves with the semigroup of a plane branch}\label{sec:monomialcurve}

Let us consider a numerical semigroup \(\Gamma,\) i.e. an additive submonoid of the natural numbers \((\Gamma,+)\subset (\mathbb{N},+)\) with finite complement \(|\mathbb{N}\setminus\Gamma|<\infty.\) Assume \(\Gamma\) is minimally generated by \(\{\overline{\beta}_0, \overline{\beta}_1, \dots, \overline{\beta}_g\}\) with \(\gcd(\overline{\beta}_0, \overline{\beta}_1, \dots, \overline{\beta}_g)=1\). Thus, \[ \Gamma = \langle \overline{\beta}_0, \overline{\beta}_1, \dots, \overline{\beta}_g \rangle=\{z\in\mathbb{N}|\;z=l_0\overline{\beta}_0+l_1\overline{\beta}_1+ \cdots+ l_g\overline{\beta}_g\;\text{and}\;l_i\in\mathbb{N}\;\text{for}\;i=0,\dots,g\}.  \]
Assume that \(\Gamma\) satisfies the following conditions:
\begin{enumerate}
	\item \(n_i\overline{\beta}_i\in\langle \overline{\beta}_0, \overline{\beta}_1, \dots, \overline{\beta}_{i-1} \rangle\),
	\item \(n_i\overline{\beta}_i<\overline{\beta}_{i+1}\) for all \(i=1,\dots,g,\)
\end{enumerate}
where \(n_{i}:=\gcd(\overline{\beta}_0, \overline{\beta}_1, \dots, \overline{\beta}_{i-1})/\gcd(\overline{\beta}_0, \overline{\beta}_1, \dots, \overline{\beta}_i)\).
A semigroup satisfying those conditions is called semigroup of a plane branch since given such a semigroup there always exists a plane branch with such a semigroup (See \cite[Chap I. 3.2]{teissier-appendix}). 

Let \(t\in\mathbb{C}\) be a local coordinate of the germ \((\mathbb{C},0)\) and let \((u_0,\dots,u_g)\in \mathbb{C}^{g+1}\) be local coordinates of the germ \((\mathbb{C}^{g+1}, \boldsymbol{0})\). Following Teissier \cite[Chap I. Sec. 1]{teissier-appendix}, let \( (C^\Gamma, \boldsymbol{0}) \subset (\mathbb{C}^{g+1}, \boldsymbol{0}) \) be the curve defined via the parameterization
\[C^{\Gamma}:\; u_i=t^{\overline{\beta}_i}\;\text{for}\;0\leq i\leq g.\]

The germ \( (C^\Gamma, \boldsymbol{0}) \) is irreducible since \( \gcd(\overline{\beta}_0, \dots, \overline{\beta}_g) = 1 \). Moreover, the monomial curve \( (C^\Gamma, \boldsymbol{0}) \) is a quasi-homogeneous complete intersection, see \cite[I.2]{teissier-appendix}. The monomial curve has the following important property:

\begin{theorem}[{\cite[Chap. I Theorem 1 (1.3)]{teissier-appendix}}]\label{thm:genericfiber}
	Every branch \( (C, \boldsymbol{0}) \) with semigroup \( \Gamma \) is isomorphic to the generic fiber of a one parameter complex analytic deformation of \( (C^\Gamma, \boldsymbol{0}) \).
\end{theorem}
\begin{rem}
	As remarked by Teissier \cite{teissier-appendix}, the above statement is a short-hand way of stating the following: For every branch \( (C, \boldsymbol{0}) \) with semigroup \( \Gamma \) there exists a deformation \(\rho: (X, \boldsymbol{0}) \longrightarrow (D, \boldsymbol{0})\) of \(C^{\Gamma}\), with a section \(\sigma\), such that for any sufficiently small representative \(\widetilde{\rho}\) of the germ of \(\rho,\) \((\widetilde{\rho}^{-1}(v),\sigma(v))\) is analytically isomorphic to \( (C, \boldsymbol{0}) \) for all \(v\neq 0\) in the image of \(\widetilde{\rho}\).
\end{rem}
\begin{rem}
	We refer to \cite{greuel-survey} for a survey about deformation theory.
\end{rem}
Let us denote by  \( G : (X, \boldsymbol{0}) \longrightarrow (S, \boldsymbol{0}) \) the miniversal deformation of \( C^\Gamma \). Let us denote by \( (C_{\boldsymbol{v}}, \boldsymbol{0}), \boldsymbol{v} \in S \) any fiber of the miniversal deformation of \( (C^\Gamma, \boldsymbol{0}) \). We will denote by \( \tau(C_{\boldsymbol{v}}) \) the dimension of the base space of the miniversal deformation of the fiber \( (C_{\boldsymbol{v}}, \boldsymbol{0}) \).  Following Teissier \cite[Chap. I, Sec. 2]{teissier-appendix}, there exists a germ of a nonsingular subspace \( (D_\Gamma,\mathbf{0})\subset(S,\boldsymbol{0}) \) such that the deformation obtained by restricting \(G\) to \(D_\Gamma\) is miniversal for the deformations of \(C^{\Gamma}\) with reduced base each of whose fibers is irreducible with semigroup \(\Gamma\) (see \cite[Chap. I, Theorem 3 (2.10)]{teissier-appendix}). Thus, \(D_\Gamma\) is called the base space of the miniversal constant semigroup deformation of \( (C^\Gamma, \boldsymbol{0}). \) Moreover, according to \cite[Chap. I, Theorem 3 (2.10)]{teissier-appendix}, if we denote the restriction of \(G\) to \(D_\Gamma\)  as \(G_\Gamma\), then there exists a section \(\sigma_\Gamma\) of \(G_\Gamma\) that picks out the unique singular point of each fiber.

The main history behind the miniversal constant semigroup deformation of \(C^{\Gamma}\) is the construction of the moduli space of branches with semigroup \(\Gamma\). Following Teissier \cite[Chap. II, Sec. 2]{teissier-appendix}, analytic equivalence of germs induces an equivalence relation \(\sim\) on \(D_\Gamma\) as follows: \(w\sim w'\) if and only if  the germs \((G_{\Gamma}^{-1}(w),\sigma_\Gamma(w))\) and \((G_{\Gamma}^{-1}(w'),\sigma_\Gamma(w'))\) are analytically isomorphic. Thus, Teissier calls \(\widetilde{M}_{\Gamma}:=D_\Gamma/\sim\) the moduli space associated to the semigroup \(\Gamma\). Moreover, \(\widetilde{M}_{\Gamma}\) is quasi-compact and connected \cite[Chap. II, Theorem 5 (2.3)]{teissier-appendix}.

One can easily check that there exist curves that are not plane in the miniversal deformation of \( (C^\Gamma, \boldsymbol{0}) \), even more  there are curves which are not plane in the miniversal constant semigroup deformation of the monomial curve. Following \cite[Chap. II 3.2]{teissier-appendix}, consider \(V_{\min}\subset D_\Gamma\) the set of points such that if \(v\in V_{\min}\) then \(\tau(C_v)=\tau_{\min}\) assumes the minimal value between all possible values of \(\tau(C_v)\) with \(v\in D_\Gamma\). The set \(V_{\min}\) is an open analytic set by \cite[Addendum, 2.5]{teissier-appendix}. Moreover, by Peraire \cite[Theorem 7.2]{peraire} together with Theorem \ref{thm:genericfiber} there exist \(v\in V_{\min}\) such that the germ \((G_{\Gamma}^{-1}(\boldsymbol{v}), \sigma_\Gamma(v))\) is an irreducible plane curve singularity with \(\tau=\tau_{\min}\). From this, we have the following corollary of Theorem \ref{teorema4/3} which gives a partial answer to Problem \ref{problem} in the case \(r=N-1\) with arbitrary \(N\). 
\begin{corollary}\label{cor:mu4space}
	Let \((C,\mathbf{0})\subset(\mathbb{C}^N,\mathbf{0})\) be an irreducible germ of curve with isolated singularity (not necessarily plane) with semigroup  \( \Gamma = \langle \overline{\beta}_0, \overline{\beta}_1, \dots, \overline{\beta}_g \rangle \) of an irreducible plane curve singularity, i.e. with semigroup satisfying conditions (1) and (2). Then, 
	
	\[\mu-\tau<\frac{\mu}{4}.\]
\end{corollary}
\begin{proof}
	Since \((C,\mathbf{0})\subset(\mathbb{C}^N,\mathbf{0})\) is an irreducible germ of curve with semigroup  \( \Gamma = \langle \overline{\beta}_0, \overline{\beta}_1, \dots, \overline{\beta}_g \rangle \) of an irreducible plane curve singularity then \((C,\mathbf{0})\) is analytically isomorphic to the generic fiber of a one parameter complex analytic deformation of \( (C^\Gamma, \boldsymbol{0}) \) by Theorem \ref{thm:genericfiber}. Let \(v\in D_\Gamma\) be such that \((C,\mathbf{0})\) is analytically isomorphic to \((G_\Gamma^{-1}(v),\sigma_\Gamma(v))\). 
	
	By \cite[Chap.I, 2.11.2]{teissier-appendix} the fibers of the miniversal constant semigroup deformation of the monomial curve \(C^{\Gamma}\) are also \(\delta(\Gamma)=|\mathbb{N}\setminus\Gamma|\)--constant. Thus, \(\delta(C)=|\mathbb{N}\setminus\Gamma|.\) Since \(C\) is an irreducible germ of curve singularity, by \cite[Proposition 1.2.1]{buchgreu} \(\mu=2\delta\). This means that 
	
	\[\mu-\tau\leq \mu-\tau_{\min}<\frac{\mu}{4},\]
	where the last inequality is coming from Theorem \ref{teorema4/3} and the fact that there exists an irreducible plane curve singularity with semigroup \(\Gamma,\) \(\tau=\tau_{\min}\) \cite{peraire} and \(\mu=2\delta\).
\end{proof}

\section{Durfee conjecture and the quotient \(\mu/\tau\) for surface singularities}\label{sec:durfee}
Following the ideas of the solution of Dimca and Greuel's Question \ref{conjecture}, we are going to continue by exploring how far the general strategy of finding optimal upper bounds for the geometric genus is useful for providing solutions to Problem \ref{problem}. In this direction, Durfee's conjecture \ref{durfeeconj} is key for our purpose.

Durfee's conjecture \ref{durfeeconj} was stated by Durfee in \cite{durfee} as a somehow natural question regarding the intersection form of the Milnor fiber. In this spirit, Durfee's conjecture has been object of an extensive study originating a strong and prolific research area. Before continuing, let's briefly sketch the state of the art of  Durfee's conjecture \ref{durfeeconj}. In the early 90s, some special cases were proven by different mathematicians: for \((X,0)\) of multiplicity 2 Tomari's Theorem \ref{thm:tomari} proves a stronger inequality \(8p_g<\mu\), for multiplicity 3 Ashikaga \cite{ashi} proves the inequality \(6p_g\leq\mu-2\), for quasi-homogeneous singularities Xu and Yau \cite{xu} prove the inequality \(6p_g\leq\mu-\operatorname{mult}(X,0)+1\). At the end of the 90s, the inequality \(6p_g\leq \mu\) is proven for the following families of surface singularities: Némethi \cite{nemethi1},\cite{nemethi2} for suspension type singularities \(\{g(x,y)+z^k=0\}\) and Melle-Hernández \cite{melle} for absolutely isolated singularities. In 2017, Kerner and Némethi \cite{kerner3} showed that Durfee's conjecture is true for Newton non-degenerate singularities with large enough Newton boundary. Recently, K\'{o}llar and Némethi prove in \cite{kollar} that Durfee's conjecture \ref{durfeeconj} is true if the link of the isolated hypersurface singularity is an integral homology sphere. Moreover, in a recent preprint \cite{eno} Enokizono show that Durfee's conjecture \ref{durfeeconj} is true whenever the topological Euler characteristic of the exceptional divisor of the minimal resolution is positive.

However, for the isolated complete intersection non-hypersurface case Kerner and N\'{e}methi show in 2012 that the inequality \(6p_g\leq\mu\) is no longer true \cite{kerner1}. They propose and they study a more general refined conjecture in a series of papers \cite{kerner1, kerner2,kerner3}:

\begin{conj}[Kerner--N\'{e}methi] \cite{kerner1, kerner2,kerner3}\label{conjeturakerner}
	Let \((X,0)\subset(\mathbb{C}^N,0)\) be an isolated complete intersection singularity of dimension \(n\) and codimension \(r=N-n.\) Then, 
	\begin{enumerate}
		\item for \(n=2\) and \(r=1\) one has \(\mu\geq 6p_g.\)
		\item for \(n=2\) and arbitrary \(r\) one has \(\mu> 4p_g.\)
		\item for \(n\geq 3\) and fixed \(r\) one has \(\mu\geq C_{n,r}p_g\) where \(C_{n,r}\) is defined by \[C_{n,r}:=\frac{\bigg(\begin{array}{c}
				n+r-1\\
				n
			\end{array}\bigg)(n+r)!}{\bigg\{\begin{array}{c}
				n+r\\
				r
			\end{array}\bigg\}r!}.\]
	\end{enumerate}
\end{conj}
Moreover, they show that those bounds are sharp.
\begin{rem}
	In fact, the bound for the hypersurface case, i.e. \((n+1)!p_g<\mu\) was already conjectured by K. Saito in 1983 \cite[Section 2 (iv), pg. 203]{saito-distribution}.
\end{rem}
Before continuing, let us introduce the following remarkable family of surface singularities.
\begin{defin}
	A surface singularity defined by the germ of function \(f:(\mathbb{C}^3,0)\rightarrow(\mathbb{C},0)\) with \(f=f_d+f_{d+1}+\cdots\) (where \(f_j\) is homogeneous of degree \(j\)) is called superisolated if the projective plane curve \(C_d:=\{f_d=0\}\subset\mathbb{P}^2\) is reduced with isolated singularities \(\{p_i\}_i\) and these points are not situated on the projective curve \(\{f_{d+1}=0\},\) where \(d\) is the degree of the initial term of \(f\).
\end{defin}
Superisolated surface singularities were first introduced by Luengo in \cite{luengomu} to show that the \(\mu\)--constant stratum in the miniversal deformation space of an isolated surface singularity, in general, is not smooth. Moreover, they are usually used to provide counterexamples to some conjectures in singularity theory. The following example shows a superisolated singularity which does not fulfill the bound \(4/3.\)
\begin{ex}
	Let us consider the superisolated surface singularity \[ f=x^{14}+y^6z^8+z^{14}+x^9z^5+(x+y+z)^{15}.\] We can compute with SINGULAR \cite{singular} that the Milnor number is \(\mu=2288\) and the Tjurina number is \(\tau=1660.\) Therefore, \(\mu/\tau>4/3.\)
	
	In this way, Theorem \ref{teorema4/3} is not true in general for surface singularities.
\end{ex} 
However, it is well known (see \cite{luengomelle}) that the geometric genus of a superisolated singularity can be expressed in terms of the degree \(d\) of the projective curve \(C_d.\) Also the Milnor number depends on the degree and of the local Milnor numbers of the singularities \(\{p_i\}_i\) of \(C_d.\)
\begin{ex}
	Let the germ of the function \(f:(\mathbb{C}^3,0)\rightarrow(\mathbb{C},0)\) with \(f=f_d+f_{d+1}+\cdots\) be a superisolated singularity. Let us denote by \(\mu_i\) the local Milnor numbers of the singular points \(\{p_i\}_i\) of the projective plane curve \(C_d:=\{f_d=0\}.\) Then we have (see \cite{luengomelle}) that
	
	\[p_g=d(d-1)(d-2)/6,\quad\mu=(d-1)^3+\sum_i\mu_i.\] 
	Therefore, it is easy to check that
	
	\[\frac{\mu}{\tau}<\frac{3}{2}.\]
\end{ex}
Also, in Wahl's paper \cite{wahl} it is given the following example which allow us to show that asymptotically \(\mu/\tau\) tends to \(3/2\) for the following family of surface singularities:
\begin{ex}
	Let us consider \(F(x,y,z)=x^d+y^d+z^d+g(x,y,z)=0\) with \(\mathrm{deg}(g)\geq d+1\). Then Example 4.7 in \cite{wahl} shows that \(\tau_{min}=(2d-3)(d+1)(d-1)/3.\) Here the minimal Tjurina number is defined as the minimal value among any Tjurina number of a positive weight deformation with fixed initial term \(x^d+y^d+z^d\).
	
	After that, it is easy to see that in this family we have
	
	\[\frac{\mu}{\tau_{\min}}\xrightarrow[d\rightarrow\infty]{}\frac{3}{2}.\]
\end{ex}
Therefore, we are under the conditions of Problem \ref{problem}. In fact, the cases where Durfee's conjecture \ref{durfeeconj} holds allow us to proof a more general result. Before stating the result, let us give first the following definitions.

\begin{defin}
	An \textit{absolutely isolated surface singularity} is a surface singularity which can be resolved after a finite number of point blowing ups.
\end{defin}
\begin{defin}
	Recall that the link \(K\) of an isolated hypersurface singularity is diffeomorphic to the boundary of the Milnor fiber. We say that the \textit{link is an integral homology sphere} if \(H_1(K;\mathbb{Z})=0.\)
\end{defin}

\begin{proposition}\label{boundsurface}
	Let \((X,0)\subset(\mathbb{C}^3,0)\) be an isolated hypersurface singularity of one of the following types:
	\begin{enumerate}
		\item[(1)] Quasi-homogeneous singularities,
		\item [(2)] \((X,0)\) of multiplicity \(3,\)
		\item [(3)] absolutely isolated singularity,
		\item [(4)] suspension of the type \(\{f(x,y)+z^N=0\},\)
		\item [(5)] the link of the singularity is an integral homology sphere,
		\item[(6)] the topological Euler characteristic of the exceptional divisor of the minimal resolution is positive.
	\end{enumerate}
	Then \[\frac{\mu}{\tau}<\frac{3}{2}.\]
\end{proposition}
\begin{proof}
	For quasi-homogeneous singularities \(\mu=\tau\) by \cite{saitohom}.
	For the cases (2), (3), (4), (5), (6) Durfee conjecture is true by \cite{ashi,melle,nemethi1,nemethi2,kollar,eno}. 
	
	Therefore by Theorem \ref{wahl} we have that for these families
	
	\[\frac{\mu}{\tau}<\frac{\mu}{\mu-2p_g}<\frac{3}{2}.\]
	
\end{proof}

Finally, since Durfee's conjecture \ref{durfeeconj} is believed to be true for hypersurface singularities, as one can see from the huge number of families for which this inequality holds, the previous discussion allows us to propose the following conjecture:
\begin{conj}\label{conjetura3/2}
	For any \((X,0)\subset(\mathbb{C}^3,0)\) isolated hypersurface singularity:
	\[\frac{\mu}{\tau}<\frac{3}{2}.\]
\end{conj}
\begin{proposition}\label{prop:durf3}
	Durfee's conjecture implies Conjecture \ref{conjetura3/2}.
\end{proposition}
\begin{proof}
	Assume Durfee's conjecture is true, \(6p_g<\mu.\) From Wahl's Theorem \ref{wahl}, we have the following bound \(\mu-\tau<2p_g<\mu/3.\) Then \(\mu/\tau<3/2.\)
\end{proof}

Despite Durfee's conjecture \ref{durfeeconj} is believed to be true and it is strongly supported, in general it is more difficult to compute the geometric genus of a family of surface singularities than its Milnor and Tjurina numbers. For this reason, Conjecture \ref{conjetura3/2} provides a good tool to check the validity of Durfee's conjecture \ref{durfeeconj} in the most complicated cases.

\printbibliography

\end{document}